\documentclass[12pt]{elsarticle}
\usepackage[T1]{fontenc}
\usepackage[utf8x]{inputenc}
\usepackage{lmodern}
\usepackage{amsmath,amsthm,amsfonts,amssymb}    
\usepackage{graphicx}
\usepackage{enumitem}
\usepackage[margin=25mm]{geometry}  
\usepackage{hyperref}   
\usepackage{stmaryrd}
\usepackage{picture}
\usepackage{mathtools}
\usepackage{tikz}   
\usepackage[french,intoc]{nomencl}    
\makenomenclature   
\usepackage{makeidx}
\makeindex
\usepackage{bbold}  
\usepackage{pgfplots}   
\pgfplotsset{compat=1.15}   
\usepackage{mathrsfs}   
\usetikzlibrary{arrows}   
\usepackage{subfigure}
\usepackage{xcolor}
\usepackage{eurosym}
\usepackage{fourier}
\usepackage{tkz-tab}    
\usepackage{diagbox}    
\usepackage{pifont}     
\allowdisplaybreaks     

\newcommand{\N}{\mathbb{N}}

\newcommand{\R}{\mathbb{R}}

\newcommand{\C}{\mathbb{C}}

\usepackage{xparse}

\NewDocumentCommand{\compactpmatrix}{m}{%
  {%
    \renewcommand{\arraystretch}{0.8}%
    \setlength{\arraycolsep}{3pt}%
    \begin{pmatrix}
      #1
    \end{pmatrix}
  }%
}

\NewDocumentCommand{\s}{m m m m m}{%
  \ensuremath{%
    \compactpmatrix{
      1 & 2 & 3 & 4 & 5 & 6 \\
      1 & #1 & #2 & #3 & #4 & #5
    }
  }%
}

\newcommand{\di}[2]{\left(a_#1^2-a_#2^2\right)}

\DeclareMathOperator{\RE}{Re}

\theoremstyle{definition}
\newtheorem{defi}{Definition}[section]  
\newtheorem{prop}[defi]{Proposition}    
\newtheorem{thm}[defi]{Theorem}
\newtheorem{lem}[defi]{Lemma}

\newtheorem{rem}[defi]{Remark}

\theoremstyle{remark}
\newtheorem*{lem*}{Lemma}
\newtheorem*{coro*}{Corollary}

\numberwithin{equation}{section}

\begin{document}

\begin{frontmatter}



\title{Minimizing numerical radius of weighted cyclic matrices under
permutation of the weights}


\author{Simon Marionnet} 
\ead{simon.marionnet@univ-lille.fr}

\affiliation{organization={Laboratoire Paul Painlevé, Université de Lille},
            city={Villeneuve d’Ascq},
            postcode={59 655 Cédex}, 
            country={France}}

\begin{abstract}
In this article we answer a question asked by Chien et al. in \cite{chien2023numerical} in which they study the numerical range of weighted cyclic matrices under permutation of their entries. Namely, we are interested in how $w(A_\sigma)$ fluctuates for various permutations $\sigma\in S_n$ and fixed $0\leq a_1<\cdots<a_n$ with $A_\sigma=\begin{pmatrix}
            0&a_{\sigma(1)}&&&\\
            &0&a_{\sigma(2)}&&\\
            &&\ddots&\ddots&\\
            &&&\ddots&a_{\sigma(n-1)}\\
            a_{\sigma(n)}&&&&0
        \end{pmatrix}$. Previous results of Gau \cite{gau2024proof} and Chang and Wang \cite{chang2012maximizing} made clear the case when $w(A_\sigma)$ is maximal among all the $w(A_\mu)$ with $\mu\in S_n$. Chien et al. in \cite{chien2023numerical} ask what the permutation which makes $w(A_\sigma)$ minimal for $n\geq 6$ could be. Answering this question is the aim of this note.
\end{abstract}




\end{frontmatter}


\section{Introduction}
Recall that for an operator $A$ acting on a complex Hilbert space (here a matrix acting on $\C^n$), the numerical range of $A$ is:
\begin{equation*}
    W(A):=\left\{\langle Ax,x\rangle\mid \lVert x\rVert=1\right\}\subset\C.
\end{equation*}
The numerical radius of $A$ is $w(A):=\sup\left\{\lvert z\rvert\mid z\in W(A)\right\}$. We know that for every operator $A$, the set $W(A)$ is convex and we have $\sigma(A)\subseteq\overline{W(A)}$ where $\sigma(A)$ is the spectrum of $A$. For the finite-dimensional case, we know that $W(A)$ is always a closed subset of the complex plane (it is the range of the unit sphere which is compact by a continuous map). Furthermore, $W(A)=W(U^*AU)$ for all unitary operator $U$. One may find all useful material about the numerical range and the numerical radius of an operator in the book \cite{gau2021numerical}.

We denote by $\RE(A)$ the real part of an operator, $\RE(A):=\frac{1}{2}(A+A^*)$.

A matrix of the shape $\begin{pmatrix}
            0&a_1&&&\\
            &0&a_2&&\\
            &&\ddots&\ddots&\\
            &&&\ddots&a_{n-1}\\
            a_n&&&&0
        \end{pmatrix}$ is called a weighted cyclic matrix. It is known from Proposition 1.1 of \cite{gau2024proof} that this is unitarily equivalent to:
        \begin{equation*}
            e^{\left(i\sum_{k=1}^n\theta_k\right)/n}\begin{pmatrix}
            0&\lvert a_1\rvert&&&\\
            &0&\lvert a_2\rvert&&\\
            &&\ddots&\ddots&\\
            &&&\ddots&\lvert a_{n-1}\rvert\\
            \lvert a_n\rvert&&&&0
        \end{pmatrix}
        \end{equation*} if we write $a_k=\lvert a_k\rvert e^{i\theta_k}$ with $\theta_k\in\R$ for all $k$. So, because multiplying an operator by a complex number of modulus 1 does not change the numerical radius, we will only consider such matrices with positive weights. However, the space on which such a matrix acts will remain $\C^n$. 

A matrix is called reducible if it is permutationaly similar to a matrix of the shape $\begin{pmatrix}
    X&Y\\
    0&Z
\end{pmatrix}$ with $X,Z$ square matrices. An irreducible matrix is a matrix that is not reducible.

For $n\in\N^*$, we set:
\begin{equation*}
    \mathcal{A}_n:=\left\{a=\left(a_1,\cdots,a_n\right)\mid 0\leq a_1<a_2<\cdots<a_n\right\}.
\end{equation*}
For $a\in\mathcal{A}_n$ and $\sigma\in S_n$ the symmetric group, we set $A_\sigma:=\begin{pmatrix}
            0&a_{\sigma(1)}&&&\\
            &0&a_{\sigma(2)}&&\\
            &&\ddots&\ddots&\\
            &&&\ddots&a_{\sigma(n-1)}\\
            a_{\sigma(n)}&&&&0
        \end{pmatrix}$. If the weights are designed by $b\in\mathcal{A}_n$ we denote by $B_\sigma$ the corresponding matrix and so on for $c,d\cdots$
        
        For $a\in\mathcal{A}_n$, we look for the permutations $\sigma\in S_n$ such that $w\left(A_{\sigma}\right)\leq w(A_\mu),\forall \mu \in S_n$.

In \cite{chang2012maximizing}, the authors denote by $H_n$ the subgroup of $S_n$ spanned by $c_n=\begin{pmatrix}
    1&2&3&\cdots&n\\
    n&1&2&\cdots&n-1
\end{pmatrix}$ and $m_n=\begin{pmatrix}
    1&2&3&\cdots&n-1&n\\
    n&n-1&n-2&\cdots&2&1
\end{pmatrix}$. They show that if $\sigma,\mu \in S_n$ are such that $\sigma=\mu h_n$ with $h_n\in H_n$, then $W(A_\sigma)=W(A_\mu)$. So we will quotient $S_n$ by $H_n$ and consider only one representative of each class of $S_n/H_n$. For $\sigma \in S_n$, composing with one of the powers of $c_n$, we may assume that $\sigma(1)=1$. Then, applying $m_n$, we may assume that $\sigma(2)<\sigma(n)$. From now on, we will just consider such permutations. We will make the identification between the classes in $S_n/H_n$ and the representatives $\sigma$ with $\sigma(1)=1$ and $\sigma(2)<\sigma(n)$ in $S_n$. It is easy to see that $H_n\cong D_n$ is the dihedral group with $2n$ elements. Thus, for this study we just need to consider $(n-1)!/2$ permutations in $S_n/H_n$ instead of $n!$ in $S_n$.

We know from \cite{chang2012maximizing} that for all $n\geq 3$ and for all $a\in\mathcal{A}_n$, there exists a unique permutation $\sigma_M^{(a)}\in S_n/H_n$ such that $w\left(A_{\sigma_M^{(a)}}\right)\geq w(A_\mu)$ for all $\mu\in S_n/H_n$. $\sigma_M^{(a)}$ is given by
\begin{equation*}
    \begin{pmatrix}
        1&\cdots&\left[\frac{n}{2}\right]-1&\left[\frac{n}{2}\right]&\left[\frac{n}{2}\right]+1&\left[\frac{n}{2}\right]+2&\left[\frac{n}{2}\right]+3&\cdots&n\\
        i&\cdots&4&2&1&3&5&\cdots&j
    \end{pmatrix}
\end{equation*}
with $(i,j)=(n-1,n)$ if $n$ is odd and $(i,j)=(n,n-1)$ if $n$ is even.

Gau proved in \cite{gau2024proof} that for all $n\in\N$ and $a\in\mathcal{A}_n$, we also have $W\left(A_\mu\right)\subseteq W\left(A_{\sigma_M^{(a)}}\right)$ for all $\mu\in S_n$.

In \cite{chien2023numerical}, the authors show that for all $a\in\mathcal{A}_4$ we have $w(A_{\sigma_4})\leq w(A_\mu)$ for all $\mu\in S_4$ if we set $\sigma_4=\compactpmatrix{
            1&2&3&4\\
            1&3&2&4}$. They show also that for all $a\in\mathcal{A}_5$, $w(A_{\sigma_5})\leq w(A_\mu)$ for all $\mu\in S_5$ with  $\sigma_5=\compactpmatrix{
            1&2&3&4&5\\
            1&4&3&2&5}$.

They ask in the remark 7.2 what can be the minimizing permutation, if it exists, for all $a\in\mathcal{A}_n$ with $n\geq6$. The answer we have is that for $n\geq6$, the situation depends on the $a$ chosen. We set
\begin{equation*}
    \mathcal{M}_n:=\left\{\sigma\in S_n/H_n\mid \exists a\in\mathcal{A}_n;\forall \mu\in S_n/H_n, w(A_\sigma)\leq w(A_\mu)\right\}.
\end{equation*}

In fact, our answer is more precise and it is given in Theorem \ref{final} as follows
\begin{equation*}
        \mathcal{M}_6=\left\{\s{5}{3}{4}{2}{6},\s{4}{5}{3}{2}{6},\s{5}{4}{3}{2}{6},\s{5}{4}{2}{3}{6},\s{4}{5}{2}{3}{6}\right\}.
\end{equation*}
In particular, the cardinal of $\mathcal{M}_6$ is greater than one in contrast with the case $n<6$.

\section{Tools for the proof}

For $a\in\mathcal{A}_n$, $A_\sigma$ is a non-negative matrix and $\RE(A)$ is irreducible. Proposition 3.3 of \cite{li2002numerical} tells us that finding the numerical radius of $A_\sigma$ is the same as finding the largest eigenvalue of the real part of $A_\sigma$. By the Perron-Froebenius theorem, this eigenvalue is simple and real. Furthermore, the associated eigenvector has only positive components. So comparing the numerical ranges of $A_\sigma$ and $A_\mu$ is the same as comparing the largest zeros of the characteristic polynomials of $\RE(2A_\sigma)$ and $\RE(2A_\mu)$. From here on, we will denote them by $\mathcal{P}_{\sigma,a}$ and $\mathcal{P}_{\mu,a}$ respectively.

\underline{As of now, we set n=6.} By the discussion about $H_n$ in the introduction, we don't need to consider the 720 permutations of $S_6$ but "only" the 60 representatives of each class in $S_6/H_6$. This is what we are going to do in the sequel.

The polynomial $\mathcal{P}_{\sigma,a}$ is computed in \cite{chien2023numerical}. For every $\sigma \in S_6$,
\begin{align*}
    \mathcal{P}_{\sigma,a}(x)&=x^6-\left(a_1^2+a_2^2+a_3^2+a_4^2+a_5^2+a_6^2\right) x^4\\
    +&\scalebox{0.8}{$\left(a_{\sigma(1)}^2a_{\sigma(3)}^2+a_{\sigma(1)}^2a_{\sigma(4)}^2+a_{\sigma(1)}^2a_{\sigma(5)}^2+a_{\sigma(2)}^2a_{\sigma(4)}^2+a_{\sigma(2)}^2a_{\sigma(5)}^2+a_{\sigma(2)}^2a_{\sigma(6)}^2+a_{\sigma(3)}^2a_{\sigma(5)}^2+a_{\sigma(3)}^2a_{\sigma(6)}^2+a_{\sigma(4)}^2a_{\sigma(6)}^2\right)$}x^2\\
    -&a_{\sigma(1)}^2a_{\sigma(3)}^2a_{\sigma(5)}^2-a_{\sigma(2)}^2a_{\sigma(4)}^2a_{\sigma(6)}^2+2a_1a_2a_3a_4a_5a_6.
\end{align*}

In the following, we will denote
\begin{align*}
    &\scalebox{0.85}{$\left(a_{\sigma(1)}^2a_{\sigma(3)}^2+a_{\sigma(1)}^2a_{\sigma(4)}^2+a_{\sigma(1)}^2a_{\sigma(5)}^2+a_{\sigma(2)}^2a_{\sigma(4)}^2+a_{\sigma(2)}^2a_{\sigma(5)}^2+a_{\sigma(2)}^2a_{\sigma(6)}^2+a_{\sigma(3)}^2a_{\sigma(5)}^2+a_{\sigma(3)}^2a_{\sigma(6)}^2+a_{\sigma(4)}^2a_{\sigma(6)}^2\right)$}\\
    &\text{ by }\left(a_{\sigma(1)}^2a_{\sigma(3)}^2+\cdots +a_{\sigma(4)}^2a_{\sigma(6)}^2\right).
\end{align*}

By replacing $x^2$ by $x$, in the expression of $\mathcal{P}_{\sigma,a}(x)$, we see that comparing the largest zero of $\mathcal{P}_{\sigma,a}(x)$ and $\mathcal{P}_{\mu,a}(x)$ boils down to comparing the largest zero of $P_{\sigma,a}(x)$ and $P_{\mu,a}(x)$, where for $\sigma \in S_6$,
\begin{align*}
    P_{\sigma,a}(x)&:=x^3-\left(a_1^2+a_2^2+a_3^2+a_4^2+a_5^2+a_6^2\right) x^2+\left(a_{\sigma(1)}^2a_{\sigma(3)}^2+\cdots +a_{\sigma(4)}^2a_{\sigma(6)}^2\right)x\\
    -&a_{\sigma(1)}^2a_{\sigma(3)}^2a_{\sigma(5)}^2-a_{\sigma(2)}^2a_{\sigma(4)}^2a_{\sigma(6)}^2+2a_1a_2a_3a_4a_5a_6.
\end{align*}

Our main tool for comparing these largest zeros will be the following simple lemma.

\begin{lem}\label{trivial}
    Let $f(t),g(t)$ be two real monic polynomials of the same degree having at least one real zero. Let's denote by $t_f$ and $t_g$, the largest real zeros of $f$ and $g$ respectively. Let $h(t)$ be defined by $h(t)=f(t)-g(t)$. If one of the following holds:
    \begin{enumerate}
        \item $h(t_f)> 0$,
        \item $h(t)>0$ for all $t\geq t_g$;
    \end{enumerate} then $t_f< t_g$.
\end{lem}

\begin{proof}
    \begin{enumerate}
        \item In this case, we have $g(t_f)=-f(t_f)+g(t_f)=-h(t_f)< 0$. Since $g(t)\to+\infty$ when $t\to+\infty$, we have $t_g> t_f$.
        \item We know that $g(t)>0$ for $t>t_g$ and by hypothesis $f(t)>g(t)$ for all $t>t_g$. So $f(t)>0$ if $t\geq t_g$, hence $t_f< t_g$.
    \end{enumerate}
\end{proof}

The authors in \cite{chien2023numerical} note that the identity $(a+b)^2=a^2+2ab+b^2$ implies:
\begin{lem}\label{id remar}
    For all $(a_1,\cdots,a_6)\in \R^6$, with $a_7=a_1$ we have:
    \begin{equation}\label{eq id remar}
        2\left(\sum_{1<\leq i<j\leq6,2\leq j-i\leq 5}a_i^2a_j^2+\sum_{k=1}^6a_k^2a_{k+1}^2\right)=\left(\sum_{i=1}^6a_i^2\right)^2-\sum_{i=1}^6a_i^4.
    \end{equation}
\end{lem}

Thus, we have the following proposition from \cite{chien2023numerical}:
\begin{prop}\label{135}
    Let $a\in\mathcal{A}_6$ and $\sigma,\mu\in S_6$ with $\{\sigma(1),\sigma(3),\sigma(5)\}=\{\mu(1),\mu(3),\mu(5)\}$. $w(A_\sigma)\leq w(A_\mu)$ if and only if $\sum_{i=1}^6 a_{\sigma(i)}^2a_{\sigma(i+1)}^2\leq \sum_{i=1}^6 a_{\mu(i)}^2a_{\mu(i+1)}^2$ where we set $a_7=a_1$.
\end{prop}

\begin{proof}
    Let's define $h$ by $h(x)=P_{\sigma,a}(x)-P_{\mu,a}(x)$. By hypothesis over $ \sigma$ and $\mu$, we have:
    \begin{equation*}
        h(x)=\left[\left(a_{\sigma(1)}^2a_{\sigma(3)}^2+\cdots +a_{\sigma(4)}^2a_{\sigma(6)}^2\right)-\left(a_{\mu(1)}^2a_{\mu(3)}^2+\cdots +a_{\mu(4)}^2a_{\mu(6)}^2\right)\right]x.
    \end{equation*}
    We know that the largest zeros of $P_{\sigma,a}$ and $P_{\mu,a}$ are strictly positive. Indeed, they are the squares of the largest zeros of $\mathcal{P}_{\sigma,a}$ and $\mathcal{P}_{\mu,a}$ respectively, which are strictly positive because they are also the numerical radii of $2A_\sigma$ and $2A_\mu$ ($a\neq 0_{\C^n}$). So, by Lemma \ref{trivial}, $w(A_\sigma)\leq w(A_\mu)$ if and only if, $\left(a_{\sigma(1)}^2a_{\sigma(3)}^2+\cdots +a_{\sigma(4)}^2a_{\sigma(6)}^2\right)\geq\left(a_{\mu(1)}^2a_{\mu(3)}^2+\cdots +a_{\mu(4)}^2a_{\mu(6)}^2\right)$.
    
    Thanks to Lemma \ref{id remar}, one can see that for all $a\in\mathcal{A}_n$ and $\sigma,\mu\in S_6$,
    \begin{equation*}
    \left(a_{\sigma(1)}^2a_{\sigma(3)}^2+\cdots +a_{\sigma(4)}^2a_{\sigma(6)}^2\right)\geq\left(a_{\mu(1)}^2a_{\mu(3)}^2+\cdots +a_{\mu(4)}^2a_{\mu(6)}^2\right)\Leftrightarrow \sum_{i=1}^6 a_{\sigma(i)}^2a_{\sigma(i+1)}^2\leq \sum_{i=1}^6 a_{\mu(i)}^2a_{\mu(i+1)}^2
    \end{equation*}
    because the right member of the equality \ref{eq id remar} does not depend on the application of a permutation $\sigma$ on the weights $a_1,\cdots,a_6$. This concludes the proof.
\end{proof}

Using an idea developed in \cite{chien2023numerical}, we are going to gather the 60 permutations in $S_6/H_6$ into 10 families with 6 permutations in each, corresponding to the 10 possible values of $\{\sigma(1),\sigma(3),\sigma(5)\}$ for $\sigma\in S_6/H_6$. These families are listed explicitly in the proof of the following proposition. Then, using Proposition \ref{135}, it will be quite easy to determine for each of the 10 families which permutation minimizes the numerical radius within the family.

The 10 families will be ordered according to the following order for the value $\{\sigma(1),\sigma(3),\sigma(5)\}$ with $\sigma\in S_6/H_6$:
\begin{equation*}
    \{1,2,3\},\{1,2,4\},\{1,2,5\},\{1,2,6\},\{1,3,4\},\{1,3,5\},\{1,3,6\},\{1,4,5\},\{1,4,6\},\{1,5,6\}.
\end{equation*}

\begin{prop}\label{Min des familles}
    Let $a\in\mathcal{A}_6$. The following permutations of $S_6$ are those which minimize the numerical radius within each of their families, respecting the precedent order of families:

$\s{5}{3}{4}{2}{6}$, 
$\s{5}{4}{3}{2}{6}$, 
$\s{4}{5}{3}{2}{6}$, 
$\s{4}{6}{3}{2}{5}$,
$\s{5}{4}{2}{3}{6}$, 

$\s{4}{5}{2}{3}{6}$, 
$\s{4}{6}{2}{3}{5}$, 
$\s{3}{5}{2}{4}{6}$, 
$\s{3}{6}{2}{4}{5}$, 
$\s{3}{6}{2}{5}{4}$.
\end{prop}

\begin{proof}
    Like in the proof of Proposition 6.5 in \cite{chien2023numerical}, for each permutation $\sigma$ we claim it is the one that minimizes $w(A_\sigma)$, we compute the value of $\sum_{i=1}^6\left(a_{\sigma(i)}^2a_{\sigma(i+1)}^2-a_{\mu(i)}^2a_{\mu(i+1)}^2\right)$ for every $\mu$ in the family of $ \sigma$. We will present the results of these simple computations in 10 tables (one per family) for easy reading. We set for a fixed  $a\in \mathcal{A}_6$,
    \begin{equation*}
        S(\sigma,\mu):=\sum_{i=1}^6\left(a_{\sigma(i)}^2a_{\sigma(i+1)}^2-a_{\mu(i)}^2a_{\mu(i+1)}^2\right).
    \end{equation*}

In the tables below, we will use the following notations. We define $r_1,\cdots,r_6$ by the equality $a_i^2=\sum_{k=1}^ir_k$ for all $i\in\llbracket1,6\rrbracket$. It is the same as asking $r_i=a_i^2-a_{i-1}^2$ for all $i\in\llbracket 1,6\rrbracket$ with $a_0=0.$ Because we have $0\leq a_1<\cdots<a_6$, we get $r_i>0$ for all $i\geq2$.

In these tables, in cases where $S(\sigma,\mu)$ is written as a sum of three terms, the decomposition of $S(\sigma,\mu)$ is obtained by replacing the differences in the sum by the corresponding values written in terms of the $r_i$'s.

\begin{center}
	\begin{tabular}{|c|c|c|}\hline
		$\mu$ & $S(\sigma,\mu)$ with $\sigma=\s{5}{3}{4}{2}{6}$&Decomposition of $S(\sigma,\mu)$\\\hline
		$\s{4}{3}{5}{2}{6}$&$\di{5}{4}\di{1}{2}$&\\\hline
		$\s{4}{2}{6}{3}{5}$&$\di{6}{4}\di{1}{3}$&\\\hline
		$\s{4}{2}{5}{3}{6}$&\scalebox{0.8}{$a_5^2\di{1}{2}+a_4^2\di{3}{1}+a_6^2\di{2}{3}$}&\scalebox{0.8}{$r_2\di{4}{5}+r_3\di{4}{6}$}\\\hline
		$\s{4}{3}{6}{2}{5}$&\scalebox{0.8}{$a_1^2\di{6}{4}+a_3^2\di{5}{6}+a_2^2\di{4}{5}$}&\scalebox{0.8}{$r_5\di{1}{2}+r_6\di{1}{3}$}\\\hline
		$\s{5}{2}{4}{3}{6}$&$\di{6}{5}\di{2}{3}$&\\\hline
	\end{tabular}
\end{center}

\begin{center}
	\begin{tabular}{|c|c|c|}\hline
		$\mu$ & $S(\sigma,\mu)$ with $\sigma=\s{5}{4}{3}{2}{6}$&Decomposition of $S(\sigma,\mu)$\\\hline
		$\s{3}{4}{6}{2}{5}$&\scalebox{0.8}{$a_1^2\di{6}{3}+a_4^2\di{5}{6}+a_2^2\di{3}{5}$}&\scalebox{0.8}{$(r_4+r_5)\di{1}{2}+r_6\di{1}{4}$}\\\hline
		$\s{3}{4}{5}{2}{6}$&$\di{5}{3}\di{1}{2}$&\\\hline
		$\s{5}{2}{3}{4}{6}$&$\di{6}{5}\di{2}{4}$&\\\hline
		$\s{3}{2}{5}{4}{6}$&\scalebox{0.8}{$a_5^2\di{1}{2}+a_3^2\di{4}{1}+a_6^2\di{2}{4}$}&\scalebox{0.8}{$r_2\di{3}{5}+(r_3+r_4)\di{3}{6}$}\\\hline
		$\s{3}{2}{6}{4}{5}$&$\di{6}{3}\di{1}{4}$&\\\hline
	\end{tabular}
\end{center}

\begin{center}
	\begin{tabular}{|c|c|c|}\hline
		$\mu$ & $S(\sigma,\mu)$ with $\sigma=\s{4}{5}{3}{2}{6}$&Decomposition of $S(\sigma,\mu)$\\\hline
		$\s{3}{5}{6}{2}{4}$&\scalebox{0.8}{$a_1^2\di{6}{3}+a_5^2\di{4}{6}+a_2^2\di{3}{4}$}&\scalebox{0.8}{$r_4\di{1}{2}+(r_5+r_6)\di{1}{5}$}\\\hline
		$\s{4}{2}{3}{5}{6}$&$\di{6}{4}\di{2}{5}$&\\\hline
		$\s{3}{2}{6}{5}{4}$&$\di{6}{3}\di{1}{5}$&\\\hline
		$\s{3}{2}{4}{5}{6}$&\scalebox{0.8}{$a_4^2\di{1}{2}+a_3^2\di{5}{1}+a_6^2\di{2}{5}$}&\scalebox{0.8}{$r_2\di{3}{4}+(r_3+r_4+r_5)\di{3}{6}$}\\\hline
		$\s{3}{5}{4}{2}{6}$&$\di{4}{3}\di{1}{2}$&\\\hline
	\end{tabular}
\end{center}

\begin{center}
	\begin{tabular}{|c|c|c|}\hline
		$\mu$ & $S(\sigma,\mu)$ with $\sigma=\s{4}{6}{3}{2}{5}$&Decomposition of $S(\sigma,\mu)$\\\hline
		$\s{4}{2}{3}{6}{5}$&$\di{5}{4}\di{2}{6}$&\\\hline
		$\s{3}{6}{4}{2}{5}$&$\di{4}{3}\di{1}{2}$&\\\hline
		$\s{3}{6}{5}{2}{4}$&\scalebox{0.8}{$a_1^2\di{5}{3}+a_6^2\di{4}{5}+a_2^2\di{3}{4}$}&\scalebox{0.8}{$r_4\di{1}{2}+r_5\di{1}{6}$}\\\hline
		$\s{3}{2}{4}{6}{5}$&\scalebox{0.8}{$a_4^2\di{1}{2}+a_3^2\di{6}{1}+a_5^2\di{2}{6}$}&\scalebox{0.8}{$r_2\di{3}{4}+(r_3+r_4+r_5+r_6)\di{3}{5}$}\\\hline
		$\s{3}{2}{5}{6}{4}$&$\di{5}{3}\di{1}{6}$&\\\hline
	\end{tabular}
\end{center}

\begin{center}
	\begin{tabular}{|c|c|c|}\hline
		$\mu$ & $S(\sigma,\mu)$ with $\sigma=\s{5}{4}{2}{3}{6}$&Decomposition of $S(\sigma,\mu)$\\\hline
		$\s{2}{4}{6}{3}{5}$&\scalebox{0.8}{$a_1^2\di{6}{2}+a_4^2\di{5}{6}+a_3^2\di{2}{5}$}&\scalebox{0.8}{$(r_3+r_4+r_5)\di{1}{3}+r_6\di{1}{4}$}\\\hline
		$\s{5}{3}{2}{4}{6}$&$\di{6}{5}\di{3}{4}$&\\\hline
		$\s{2}{3}{6}{4}{5}$&$\di{6}{2}\di{1}{4}$&\\\hline
		$\s{2}{3}{5}{4}{6}$&\scalebox{0.8}{$a_5^2\di{1}{3}+a_2^2\di{4}{1}+a_6^2\di{3}{4}$}&\scalebox{0.8}{$(r_2+r_3)\di{2}{5}+r_4\di{2}{6}$}\\\hline
		$\s{2}{4}{5}{3}{6}$&$\di{5}{2}\di{1}{3}$&\\\hline
	\end{tabular}
\end{center}

\begin{center}
	\begin{tabular}{|c|c|c|}\hline
		$\mu$ & $S(\sigma,\mu)$ with $\sigma=\s{4}{5}{2}{3}{6}$&Decomposition of $S(\sigma,\mu)$\\\hline
		$\s{2}{5}{4}{3}{6}$&$\di{4}{2}\di{1}{3}$&\\\hline
		$\s{2}{3}{4}{5}{6}$&\scalebox{0.8}{$a_4^2\di{1}{3}+a_2^2\di{5}{1}+a_6^2\di{3}{5}$}&\scalebox{0.8}{$(r_2+r_3)\di{2}{4}+(r_4+r_5)\di{2}{6}$}\\\hline
		$\s{4}{3}{2}{5}{6}$&$\di{6}{4}\di{3}{5}$&\\\hline
		$\s{2}{3}{6}{5}{4}$&$\di{6}{2}\di{1}{5}$&\\\hline
		$\s{2}{5}{6}{3}{4}$&\scalebox{0.8}{$a_1^2\di{6}{2}+a_5^2\di{4}{6}+a_3^2\di{2}{4}$}&\scalebox{0.8}{$(r_3+r_4)\di{1}{3}+(r_5+r_6)\di{1}{5}$}\\\hline
	\end{tabular}
\end{center}

\begin{center}
	\begin{tabular}{|c|c|c|}\hline
		$\mu$ & $S(\sigma,\mu)$ with $\sigma=\s{4}{6}{2}{3}{5}$&Decomposition of $S(\sigma,\mu)$\\\hline
		$\s{4}{3}{2}{6}{5}$&$\di{5}{4}\di{3}{6}$&\\\hline
		$\s{2}{3}{4}{6}{5}$&\scalebox{0.8}{$a_4^2\di{1}{3}+a_2^2\di{6}{1}+a_5^2\di{3}{6}$}&\scalebox{0.8}{$(r_2+r_3)\di{2}{4}+(r_4+r_5+r_6)\di{2}{5}$}\\\hline
		$\s{2}{6}{5}{3}{4}$&\scalebox{0.8}{$a_1^2\di{5}{2}+a_6^2\di{4}{5}+a_3^2\di{2}{4}$}&\scalebox{0.8}{$(r_3+r_4)\di{1}{3}+r_5\di{1}{6}$}\\\hline
		$\s{2}{6}{4}{3}{5}$&$\di{4}{2}\di{1}{3}$&\\\hline
		$\s{2}{3}{5}{6}{4}$&$\di{5}{2}\di{1}{6}$&\\\hline
	\end{tabular}
\end{center}

\begin{center}
	\begin{tabular}{|c|c|c|}\hline
		$\mu$ & $S(\sigma,\mu)$ with $\sigma=\s{3}{5}{2}{4}{6}$&Decomposition of $S(\sigma,\mu)$\\\hline
		$\s{2}{5}{3}{4}{6}$&$\di{3}{2}\di{1}{4}$&\\\hline
		$\s{2}{5}{6}{4}{3}$&\scalebox{0.8}{$a_1^2\di{6}{2}+a_5^2\di{3}{6}+a_4^2\di{2}{3}$}&\scalebox{0.8}{$r_3\di{1}{4}+(r_4+r_5+r_6)\di{1}{5}$}\\\hline
		$\s{2}{4}{3}{5}{6}$&\scalebox{0.8}{$a_3^2\di{1}{4}+a_2^2\di{5}{1}+a_6^2\di{4}{5}$}&\scalebox{0.8}{$(r_2+r_3+r_4)\di{2}{3}+r_5\di{2}{6}$}\\\hline
		$\s{2}{4}{6}{5}{3}$&$\di{6}{2}\di{1}{5}$&\\\hline
		$\s{3}{4}{2}{5}{6}$&$\di{5}{4}\di{3}{6}$&\\\hline
	\end{tabular}
\end{center}

\begin{center}
	\begin{tabular}{|c|c|c|}\hline
		$\mu$ & $S(\sigma,\mu)$ with $\sigma=\s{3}{6}{2}{4}{5}$&Decomposition of $S(\sigma,\mu)$\\\hline
		$\s{2}{4}{5}{6}{3}$&$\di{5}{2}\di{1}{6}$&\\\hline
		$\s{2}{6}{5}{4}{3}$&\scalebox{0.8}{$a_1^2\di{5}{2}+a_6^2\di{3}{5}+a_4^2\di{2}{3}$}&\scalebox{0.8}{$r_3\di{1}{4}+(r_4+r_5)\di{1}{6}$}\\\hline
		$\s{2}{4}{3}{6}{5}$&\scalebox{0.8}{$a_3^2\di{1}{4}+a_2^2\di{6}{1}+a_5^2\di{4}{6}$}&\scalebox{0.8}{$(r_2+r_3+r_4)\di{2}{3}+(r_5+r_6)\di{2}{5}$}\\\hline
		$\s{3}{4}{2}{6}{5}$&$\di{6}{4}\di{3}{5}$&\\\hline
		$\s{2}{6}{3}{4}{5}$&$\di{3}{2}\di{1}{4}$&\\\hline
	\end{tabular}
\end{center}

\begin{center}
	\begin{tabular}{|c|c|c|}\hline
		$\mu$ & $S(\sigma,\mu)$ with $\sigma=\s{3}{6}{2}{5}{4}$&Decomposition of $S(\sigma,\mu)$\\\hline
		$\s{2}{5}{4}{6}{3}$&$\di{4}{2}\di{1}{6}$&\\\hline
		$\s{3}{5}{2}{6}{4}$&$\di{6}{5}\di{3}{4}$&\\\hline
		$\s{2}{5}{3}{6}{4}$&\scalebox{0.8}{$a_3^2\di{1}{5}+a_2^2\di{6}{1}+a_4^2\di{5}{6}$}&\scalebox{0.8}{$(r_2+r_3+r_4+r_5)\di{2}{3}+r_6\di{2}{4}$}\\\hline
		$\s{2}{6}{4}{5}{3}$&\scalebox{0.8}{$a_1^2\di{4}{2}+a_6^2\di{3}{4}+a_5^2\di{2}{3}$}&\scalebox{0.8}{$r_3\di{1}{5}+r_4\di{1}{6}$}\\\hline
		$\s{2}{6}{3}{5}{4}$&$\di{3}{2}\di{1}{5}$&\\\hline
	\end{tabular}
\end{center}

    The result of each of these calculations is negative. So, by Proposition \ref{135}, the proof is complete.
\end{proof}

For an $a\in\mathcal{A}_n$ given, we now "just" have to compare the numerical radii given by these 10 permutations to find the permutation(s) that will minimize the latter.

\begin{prop}\label{delta}
    Let $a\in\mathcal{A}_6$ and $\sigma,\mu\in S_6$. Let
    \begin{equation*}
        \alpha:=\left(a_{\sigma(1)}^2a_{\sigma(3)}^2+\cdots +a_{\sigma(4)}^2a_{\sigma(6)}^2\right)-\left(a_{\mu(1)}^2a_{\mu(3)}^2+\cdots +a_{\mu(4)}^2a_{\mu(6)}^2\right)
    \end{equation*}
    and 
    \begin{equation*}
        \beta:=-a_{\sigma(1)}^2a_{\sigma(3)}^2a_{\sigma(5)}^2-a_{\sigma(2)}^2a_{\sigma(4)}^2a_{\sigma(6)}^2+a_{\mu(1)}^2a_{\mu(3)}^2a_{\mu(5)}^2+a_{\mu(2)}^2a_{\mu(4)}^2a_{\mu(6)}^2.
    \end{equation*}
    We have
    \begin{equation*}
        \Delta:=\left(\sum_{i=1}^6a_i^2\right)^2-3\left(a_{\sigma(1)}^2a_{\sigma(3)}^2+\cdots +a_{\sigma(4)}^2a_{\sigma(6)}^2\right)>0.
    \end{equation*}
    In addition, with
    \begin{equation*}
        x_M:=\frac{1}{3}\left(\sum_{i=1}^6a_i^2+\sqrt{\Delta}\right),
    \end{equation*}
    we have the following:
    \begin{enumerate}
        \item If $\alpha=0$:
        \begin{enumerate}
            \item if $\beta>0$, $w(A_\mu)>w(A_\sigma)$,
            \item if $\beta<0$, $w(A_\sigma)>w(A_\mu)$.
        \end{enumerate}
        \item If $\alpha<0$ and $\gamma:=-\beta/\alpha\leq x_M$, then $w(A_\sigma)>w(A_\mu)$.
        \item If $\alpha>0$ and $\gamma:=-\beta/\alpha\leq x_M$, then $w(A_\mu)>w(A_\sigma)$.
    \end{enumerate}
\end{prop}

\begin{proof}
    If $\alpha=0$ and $\beta>0$, we have $P_{\sigma,a}(t)>P_{\mu,a}(t)$ for all $t>0$. Lemma \ref{trivial} gives the desired conclusion and similarly for $\alpha=0$ and $\beta<0$.

    We have:
    \begin{equation*}
        P'_{\sigma,a}(x)=3x^2-2\left(a_1^2+a_2^2+a_3^2+a_4^2+a_5^2+a_6^2\right) x+\left(a_{\sigma(1)}^2a_{\sigma(3)}^2+\cdots +a_{\sigma(4)}^2a_{\sigma(6)}^2\right).
    \end{equation*}
    The discriminant of this polynomial is $4\Delta$.

    We know that the roots of $P_{\sigma,a}$ are the squares of the roots of $\mathcal{P}_{\sigma,a}$, which are exactly the eigenvalues of 2$\RE(A_\sigma)$. The fact that $\RE(A_\sigma)$ is a non-negative irreducible matrix and the Perron-Froebenius theorem insures that the roots of $\mathcal{P}_{\sigma,a}$ are real and the largest root of $\mathcal{P}_{\sigma,a}$ is simple. Hence, because of the parity of $\mathcal{P}_{\sigma,a}$, there exists $\lambda_1\geq\lambda_2\geq\lambda_3\geq0$ such that the roots of $\mathcal{P}_{\sigma,a}$ are the $\pm\lambda_i, i=1,2,3$. Because $\lambda_1$ is a simple root of $\mathcal{P}_{\sigma,a}$, we have $\lambda_1>\lambda_2$. So $P_{\sigma,a}$ has at least two real roots, so it must therefore have three because it is a real polynomial of degree three. Thus, we must have $\Delta>0$, and $P'_{\sigma,a}$ has two real roots given by:
    \begin{equation*}
        x_m=\frac{1}{3}\left(\sum_{i=1}^6a_i^2-\sqrt{\Delta}\right)\text{ and }
        x_M=\frac{1}{3}\left(\sum_{i=1}^6a_i^2+\sqrt{\Delta}\right).
    \end{equation*}
    We also have $P_{\sigma,a}(x_M)<0$ because $P_{\sigma,a}$ has three real roots.
    
    So if we denote by $t_\sigma$ and $t_\mu$ respectively the largest zeros of $P_{\sigma,a}$ and $P_{\mu,a}$, we have $x_M<t_\sigma$. We now suppose $-\beta/\alpha\leq x_M$. Applying the first part of Lemma \ref{trivial} to $(f,g)=\left(P_{\sigma,a},P_{\mu,a}\right)$ if $\alpha>0$ or the second part of Lemma \ref{trivial} to $(f,g)=\left(P_{\mu,a}~ ,P_{\sigma,a}\right)$ if $\alpha<0$ gives the desired conclusion.
\end{proof}

\begin{prop}\label{loc rayon num}
    Let $a\in\mathcal{A}_6$, and let $\sigma\in S_6$. If we set $G_\sigma:=\max_{i\in\llbracket 1,n\rrbracket}\left\{a_{\sigma(i)}+a_{\sigma(i+1)}\right\}$  and $g_\sigma:=\min_{i\in\llbracket 1,n\rrbracket}\left\{a_{\sigma(i)}+a_{\sigma(i+1)}\right\}$ with $a_{\sigma(7)}=a_{\sigma(1)}$, then:
    \begin{equation*}
        \frac{g_\sigma}{2}\leq w(A_\sigma)\leq \frac{G_\sigma}{2}.
    \end{equation*}
\end{prop}

\begin{proof}
    Let's denote here by $t_\sigma$ the largest zero of $\mathcal{P}_{\sigma,a}$ , then $t_\sigma$ is the largest eigenvalue of $2\RE (A_\sigma)$ and $t_\sigma=w(2A_\sigma)$. Also, we have:
    \begin{equation*}
        2\RE (A_\sigma)=\begin{pmatrix}
            0&a_{\sigma(1)}&&&a_{\sigma(n)}\\
            a_{\sigma(1)}&0&a_{\sigma(2)}&&\\
            &\ddots&\ddots&\ddots&\\
            &&\ddots&\ddots&a_{\sigma(n-1)}\\
            a_{\sigma(n)}&&&a_{\sigma(n-1)}&0
        \end{pmatrix}.
    \end{equation*}
    Hence its disks of Gershgorin are the $\left(\overline{D}(0,a_{\sigma(i)}+a_{\sigma(i+1)})\right)_{i\in\llbracket 1,n\rrbracket}$ with $a_{\sigma(7)}=a_{\sigma(1)}$. So if we set $G_\sigma$ as in the statement of the proposition, we know that the largest eigenvalue of $2\RE (A_\sigma)$ belongs to $]0,G_\sigma]$.

    On the other hand, we know by the Perron-Froebenius theorem that $t_\sigma$ is an eigenvalue of $2\RE(A_\sigma)$ with associated eigenvector $x=\left(x_i\right)_{i\in\llbracket 1,6\rrbracket}$ such that $x_i>0,\forall i$. If we denote $i_0\in\llbracket1,6\rrbracket$ the index such that $x_{i_0}=\min\{x_i\mid i=1,\cdots,6\}$, we have $x_{i_0}>0$ and since $t_\sigma x=2\RE(A_\sigma)x$ we have:
    \begin{equation*}
        t_\sigma x_{i_0}=a_{\sigma(i_0-1)}x_{\sigma(i_0-1)}+a_{\sigma(i_0)}x_{\sigma(i_0+1)}.
    \end{equation*}
    So by choice of $i_0$, $t_\sigma\geq a_{\sigma(i_0-1)}+a_{\sigma(i_0)}\geq g_\sigma$.
\end{proof}

\begin{rem}\label{simple impossible}
    Let $a\in\mathcal{A}_6,\sigma,\mu\in S_6$. An immediate consequence of Proposition \ref{loc rayon num} is that if $G_\sigma<g_\mu$, then $w(A_\sigma)<w(A_\mu)$ with the same notations as above. Unfortunately, one can never have the condition $G_\sigma<g_\mu$ irrespective of the $a\in \mathcal{A}_6$, and $\sigma,\mu \in S_6$ chosen. Indeed, if we set $a\in\mathcal{A}_6$ and consider $\sigma,\mu\in S_6$ that we will fix later, in order to perhaps have $G_\sigma<g_\mu$, our goal is to achieve $g_\mu$ as big as possible and $G_\sigma$ as small as possible.

    If two permutations $\sigma_1,\sigma_2\in S_6$ are in the same class in $S_6/H_6$, we have $g_{\sigma_1}=g_{\sigma_2}$ and $G_{\sigma_1}=G_{\sigma_2}$. So, by applying the permutations $c_6$ and $m_6$ as many times as needed, to seek the maximal $g_\mu$, we will consider what happens if $\mu$ is in the same class as $\tilde{\mu}$ in $S_6/H_6$ (so not necessary with $\tilde{\mu}(2)<\tilde{\mu}(6)$ as for $\mu$) with $\tilde{\mu}$ being one of the following type:
    \begin{enumerate}
        \item If $\tilde{\mu}=\s{2}{i}{j}{k}{l}$, then $g_\mu=a_1+a_2$.
        \item  If $\tilde{\mu}=\s{i}{2}{j}{k}{l}$, then $g_\mu=\min\left\{a_1+a_i,a_2+a_j,a_1+a_l\right\}$.
        \item If $\tilde{\mu}=\s{i}{j}{2}{k}{l}$, then $g_\mu=\min\left\{a_1+a_i,a_1+a_l,a_2+a_j,a_2+a_k\right\}$,
    \end{enumerate}
    with $\{i,j,k,l\}=\{3,4,5,6\}$. The case when $g_\mu$ is maximal is when $\tilde{\mu}=\s{6}{2}{4}{3}{5}$ and so $\mu=\s{5}{3}{4}{2}{6}$ and $g_\mu=\min\{a_1+a_5,a_2+a_4\}$. 

    By the symmetry of the role played respectively by 1,2 and 6,5 in the arguments above, the case where $G_\sigma$ is the smaller is when $\sigma$ is in the same class in $S_6/H_6$ as $\tilde{\sigma}=\compactpmatrix{
        1&2&3&4&5&6\\
        6&1&5&3&4&2}$. In this case, $\sigma$ is also equal to $\s{5}{3}{4}{2}{6}$. Clearly, we have not $G_\sigma<g_\mu$ in this case. So for worse $\sigma$ and $\mu$ (\textit{ie.} for $\sigma',\mu'$ such that $g_{\mu'}\leq g_\mu$ and $G_{\sigma'}\geq G_\sigma$) this will not happen either.

    So, for this reason, we need a more accurate tool, which will be provided by the following proposition.
\end{rem}

\begin{prop}\label{Gershgo}
    Let $a\in\mathcal{A}_6$ and $\sigma,\mu\in S_6$. Let 
    \begin{equation*}
        G_\sigma:=\max_{i\in\llbracket 1,n\rrbracket}\left\{a_{\sigma(i)}+a_{\sigma(i+1)}\right\}\text{ and }g_\sigma:=\min_{i\in\llbracket 1,n\rrbracket}\left\{a_{\sigma(i)}+a_{\sigma(i+1)}\right\},
    \end{equation*}
    with $a_{\sigma(7)}=a_{\sigma(1)}$. We set
    \begin{equation*}
        \alpha:=\left(a_{\sigma(1)}^2a_{\sigma(3)}^2+\cdots +a_{\sigma(4)}^2a_{\sigma(6)}^2\right)-\left(a_{\mu(1)}^2a_{\mu(3)}^2+\cdots +a_{\mu(4)}^2a_{\mu(6)}^2\right)
    \end{equation*}
    and 
    \begin{equation*}
        \beta:=-a_{\sigma(1)}^2a_{\sigma(3)}^2a_{\sigma(5)}^2-a_{\sigma(2)}^2a_{\sigma(4)}^2a_{\sigma(6)}^2+a_{\mu(1)}^2a_{\mu(3)}^2a_{\mu(5)}^2+a_{\mu(2)}^2a_{\mu(4)}^2a_{\mu(6)}^2.
    \end{equation*}
    We suppose $\alpha\neq0$ and we set $\gamma:=-\beta/\alpha$. If we have one of the following:
    \begin{enumerate}
        \item $\alpha<0$ and $\gamma>G_\sigma^2$,
        \item $\alpha>0$ and $\gamma<g_\sigma^2$;
    \end{enumerate}
    then $w(A_\mu)>w(A_\sigma)$.
\end{prop}

\begin{proof}
    Both cases insure that $P_{\sigma,a}(t)>P_{\mu,a}(t)$ if $t\in [g_\sigma^2,G_\sigma^2]$. Hence by Proposition \ref{loc rayon num}, we know that $P_{\sigma,a}(t_\sigma)>P_{\mu,a}(t_\sigma)$ if $t_\sigma$ denotes here the largest zero of $P_{\sigma,a}$ because then $t_\sigma=4w(A_\sigma)^2\in [g_\sigma^2,G_\sigma^2]$. Lemma \ref{trivial} gives the targeted conclusion.
\end{proof}

\section{Proof of the theorem}

\begin{thm}\label{Il en reste 5}
    We have:
    \begin{equation*}
        \mathcal{M}_6\subseteq\left\{\s{5}{3}{4}{2}{6},\s{4}{5}{3}{2}{6},\s{5}{4}{3}{2}{6},\s{5}{4}{2}{3}{6},\s{4}{5}{2}{3}{6}\right\}.
    \end{equation*}
\end{thm}

\begin{proof}
    Remembering Proposition \ref{Min des familles}, it is sufficient to prove that the following permutations are not in $\mathcal{M}_6$:
    \begin{equation*}
        \s{4}{6}{3}{2}{5},\s{4}{6}{2}{3}{5},\s{3}{6}{2}{5}{4},\s{3}{5}{2}{4}{6},\s{3}{6}{2}{4}{5}.
    \end{equation*}
    Fix $a\in\mathcal{A}_6$ for the rest of the proof. We want to show that for each of these five preceding permutations, there exists another permutation in $S_n/H_n$ which leads to a smaller numerical radius. To do so, we will first apply Proposition \ref{Gershgo} with $\sigma=\s{4}{5}{2}{3}{6}$ and $\mu$ equals to $\s{4}{6}{3}{2}{5}$ and $\s{4}{6}{2}{3}{5}$ successively. We will keep the notations of \ref{Gershgo} throughout the proof.

    \ding{172} If $\sigma=\s{4}{5}{2}{3}{6}$ and $\mu=\s{4}{6}{3}{2}{5}$, we have
    \begin{align*}
        \alpha&=\left( a_1^2a_5^2+a_1^2a_2^2+a_1^2a_3^2+a_4^2a_2^2+a_4^2a_3^2+a_4^2a_6^2+a_5^2a_3^2+a_5^2a_6^2+a_2^2a_6^2 \right)\\
        &-\left( a_1^2a_6^2+a_1^2a_3^2+a_1^2a_2^2+a_4^2a_3^2+a_4^2a_2^2+a_4^2a_5^2+a_6^2a_2^2+a_6^2a_5^2+a_3^2a_5^2 \right)\\
        &=a_1^2a_5^2+a_4^2a_6^2-a_1^2a_6^2-a_4^2a_5^2=\left(a_1^2-a_4^2\right)\left(a_5^2-a_6^2\right)>0.
    \end{align*}
    On the other hand, we have:
    \begin{equation*}
        \beta=-a_1^2a_5^2a_3^2-a_4^2a_2^2a_6^2+a_1^2a_6^2a_2^2+a_4^2a_3^2a_5^2=a_2^2a_6^2\left(a_1^2-a_4^2\right)+a_5^2a_3^2\left(a_4^2-a_1^2\right)=\left(a_4^2-a_1^2\right)\left(a_5^2a_3^2-a_2^2a_6^2\right).
    \end{equation*}
    So we have $\gamma=\frac{-\beta}{\alpha}=\frac{a_5^2a_3^2-a_2^2a_6^2}{a_5^2-a_6^2}$. In this case, $g_\sigma$ can be $a_1+a_4$ or $a_2+a_3$.
    
    \underline{If $g_\sigma=a_1+a_4$}, 
    \begin{align*}
        \gamma<g_\sigma^2&\Leftrightarrow \frac{a_5^2a_3^2-a_2^2a_6^2}{a_5^2-a_6^2}<(a_1+a_4)^2\\
        &\Leftrightarrow a_6^2a_2^2-a_5^2a_3^2<\left(a_6^2-a_5^2\right)(a_1+a_4)^2\\
        &\Leftrightarrow a_6^2\left[(a_1+a_4)^2-a_2^2\right]>a_5^2\left[(a_1+a_4)^2-a_3^2\right].
    \end{align*}
    Because $a_3>a_2$, $a_6>a_5$ and the fact that $(a_1+a_4)^2-a_2^2>0$, we always have $a_6^2\left[(a_1+a_4)^2-a_2^2\right]>a_5^2\left[(a_1+a_4)^2-a_3^2\right]$. So we have $\gamma<g_\sigma^2$ and by Proposition \ref{Gershgo} ($\alpha>0$), we obtain $w(A_\mu)>w(A_\sigma)$.
    \underline{Otherwise, if $g_\sigma=a_2+a_3$}, we have by the same computations:
    \begin{equation*}
        \gamma<g_\sigma^2\Leftrightarrow a_6^2\left[(a_2+a_3)^2-a_2^2\right]>a_5^2\left[(a_2+a_3)^2-a_3^2\right].
    \end{equation*}
    Since $(a_2+a_3)^2-a_2^2>0$, $\gamma<g_\sigma^2$ and so we can apply Proposition \ref{Gershgo} to show $w(A_\mu)>w(A_\sigma)$. We can now claim that $\s{4}{6}{3}{2}{5}\not\in\mathcal{M}_6$.

    \ding{173} Now, we set $\sigma=\s{4}{5}{2}{3}{6}$ and $\mu=\s{4}{6}{2}{3}{5}$. We have:
    \begin{align*}
        \alpha&=\left( a_1^2a_5^2+a_1^2a_2^2+a_1^2a_3^2+a_4^2a_2^2+a_4^2a_3^2+a_4^2a_6^2+a_5^2a_3^2+a_5^2a_6^2+a_2^2a_6^2 \right)\\
        &-\left( a_1^2a_6^2+a_1^2a_2^2+a_1^2a_3^2+a_4^2a_2^2+a_4^2a_3^2+a_4^2a_5^2+a_6^2a_3^2+a_6^2a_5^2+a_2^2a_5^2 \right)\\
        &=a_1^2a_5^2+a_4^2a_6^2+a_5^2a_3^2+a_2^2a_6^2-a_1^2a_6^2-a_4^2a_5^2-a_6^2a_3^2- a_2^2a_5^2\\
        &=\left(a_6^2-a_5^2\right)\left(a_4^2+a_2^2-a_3^2-a_1^2\right)>0.
    \end{align*}
    
    On the other hand, we have $\beta=\left(a_6^2-a_5^2\right)\left(a_1^2a_3^2-a_2^2a_4^2\right)$. Hence we have $\gamma=\frac{a_2^2a_4^2-a_1^2a_3^2}{a_4^2+a_2^2-a_3^2-a_1^2}$.

    So, \underline{if $g_\sigma=a_1+a_4$},
    \begin{align*}
        \gamma<g_\sigma^2&\Leftrightarrow \frac{a_2^2a_4^2-a_1^2a_3^2}{a_4^2+a_2^2-a_3^2-a_1^2}<(a_1+a_4)^2\\
        &\Leftrightarrow (a_1+a_4)^2\left(a_4^2+a_2^2-a_3^2-a_1^2\right)-a_2^2a_4^2+a_1^2a_3^2>0\\
        &\Leftrightarrow a_1^2\left(a_2^2-a_1^2\right)+a_4^2\left(a_4^2-a_3^2\right)+2a_1a_4\left(a_2^2+a_4^2-a_3^2-a_1^2\right)>0.
    \end{align*}
    This latter condition is always true because $a_4>a_3$ and $a_2>a_1$. So we always have $\gamma <g_\sigma^2$, hence $w(A_\mu)>w(A_\sigma)$.

    \underline{Otherwise, if $g_\sigma=a_2+a_3$}, by almost the same computations, we have:
    \begin{align*}
        \gamma<g_\sigma^2&\Leftrightarrow \frac{a_2^2a_4^2-a_1^2a_3^2}{a_4^2+a_2^2-a_3^2-a_1^2}<(a_2+a_3)^2\\
        &\Leftrightarrow (a_2+a_3)^2\left(a_4^2+a_2^2-a_3^2-a_1^2\right)-a_2^2a_4^2+a_1^2a_3^2>0\\
        &\Leftrightarrow a_2^2\left(a_2^2-a_1^2\right)+a_3^2\left(a_4^2-a_3^2\right)+2a_2a_3\left(a_2^2+a_4^2-a_3^2-a_1^2\right)>0.
    \end{align*}
    By the same arguments, we see that the last condition always holds. So we can also apply Proposition \ref{Gershgo}. We can now claim that $\s{4}{6}{2}{3}{5}\not\in\mathcal{M}_6$.

    \ding{174} If then $\sigma=\s{3}{6}{2}{4}{5}$ and $\mu=\s{3}{6}{2}{5}{4}$, we have:
    \begin{align*}
        \alpha&=\left(a_1^2a_6^2+a_1^2a_2^2+a_1^2a_4^2+a_3^2a_2^2+a_3^2a_4^2+a_3^2a_5^2+a_6^2a_4^2+a_6^2a_5^2+a_2^2a_5^2\right)\\
        &-\left(a_1^2a_6^2+a_1^2a_2^2+a_1^2a_5^2+a_3^2a_2^2+a_3^2a_5^2+a_3^2a_4^2+a_6^2a_5^2+a_6^2a_4^2+a_2^2a_4^2\right)\\
        &=a_1^2a_4^2+a_2^2a_5^2-a_1^2a_5^2-a_2^2a_4^2=\left(a_1^2-a_2^2\right)\left(a_4^2-a_5^2\right)>0.
    \end{align*}
    We also have:
    \begin{equation*}
        \beta=-a_1^2a_6^2a_4^2-a_3^2a_2^2a_5^2+a_1^2a_6^2a_5^2+a_3^2a_2^2a_4^2=\left(a_5^2-a_4^2\right)\left(a_1^2a_6^2-a_3^2a_2^2\right).
    \end{equation*}
    Hence we have $\gamma=\frac{\left(a_1^2a_6^2-a_3^2a_2^2\right)}{\left(a_1^2-a_2^2\right)}$. On the other hand, $g_\sigma=a_1+a_3$. So,
    \begin{align*}
        \gamma<g_\sigma^2&\Leftrightarrow \frac{\left(a_1^2a_6^2-a_3^2a_2^2\right)}{\left(a_1^2-a_2^2\right)}<(a_1+a_3)^2\\
        &\Leftrightarrow (a_1+a_3)^2\left(a_2^2-a_1^2\right)>a_3^2a_2^2-a_1^2a_6^2\\
        &\Leftrightarrow a_2^2\left[(a_1+a_3)^2-a_3^2\right]>a_1^2\left[(a_1+a_3)^2-a_6^2\right]
    \end{align*}
    The latter is always true because $(a_1+a_3)^2-a_3^2>0$, $a_2>a_1$ and $a_6>a_3$. Thus, the condition $\gamma<g_\sigma^2$ holds. Hence, by Proposition \ref{Gershgo}, we have $\s{3}{6}{2}{5}{4}\not\in\mathcal{M}_6$.

    Now, in order to eliminate the two remaining permutations, we want to apply Proposition \ref{delta} with\\ $\sigma=\s{4}{5}{2}{3}{6}$ and $\mu$ equals successively $\s{3}{5}{2}{4}{6}$ and $\s{3}{6}{2}{4}{5}$.

    \ding{175} Let $\sigma=\s{4}{5}{2}{3}{6}$ and $\mu=\s{3}{5}{2}{4}{6}$. We have:
    \begin{align*}
        \alpha&=a_1^2a_5^2+a_1^2a_2^2+a_1^2a_3^2+a_4^2a_2^2+a_4^2a_3^2+a_4^2a_6^2+a_5^2a_3^2+a_5^2a_6^2+a_2^2a_6^2\\
        &-\left(a_1^2a_5^2+a_1^2a_2^2+a_1^2a_4^2+a_3^2a_2^2+a_3^2a_4^2+a_3^2a_6^2+a_5^2a_4^2+a_5^2a_6^2+a_2^2a_6^2\right)\\
        &=a_1^2a_3^2+a_4^2a_2^2+a_4^2a_6^2+a_5^2a_3^2-a_1^2a_4^2-a_3^2a_2^2-a_3^2a_6^2-a_5^2a_4^2\\
        &=a_3^2\left(a_1^2+a_5^2-a_2^2-a_6^2\right)+a_4^2\left(a_2^2+a_6^2-a_1^2-a_5^2\right)\\
        &=\left(a_4^2-a_3^2\right)\left(a_2^2+a_6^2-a_1^2-a_5^2\right)>0.
\end{align*}
We also have:
\begin{align*}
    \beta&=-a_1^2a_5^2a_3^2-a_4^2a_2^2a_6^2+a_1^2a_5^2a_4^2+a_3^2a_2^2a_6^2\\
    &=a_3^2\left(a_2^2a_6^2-a_1^2a_5^2\right)+a_4^2\left(a_1^2a_5^2-a_2^2a_6^2\right)\\
    &=\left(a_3^2-a_4^2\right)\left(a_2^2a_6^2-a_1^2a_5^2\right).
\end{align*}
Hence $\gamma=\frac{-\beta}{\alpha}=\frac{a_2^2a_6^2-a_1^2a_5^2}{a_2^2+a_6^2-a_1^2-a_5^2}$. If we set as in Proposition \ref{delta}
\begin{equation*}
    \Delta:=\left(\sum_{i=1}^6a_i^2\right)^2-3\left(a_{\sigma(1)}^2a_{\sigma(3)}^2+\cdots +a_{\sigma(4)}^2a_{\sigma(6)}^2\right),
\end{equation*}
in view of Lemma \ref{id remar}, we have:
\begin{equation*}
    \Delta=\sum_{i=1}^6a_i^4+2\sum_{k=1}^6a_{\sigma(k)}^2a_{\sigma(k+1)}^2-\left(a_{\sigma(1)}^2a_{\sigma(3)}^2+\cdots +a_{\sigma(4)}^2a_{\sigma(6)}^2\right).
\end{equation*}
With the same notations as in Proposition \ref{delta}, we have:
\begin{align*}
    \gamma\leq x_M&\Leftrightarrow \left(\sum_{i=0}^6a_i^2+ \sqrt{\Delta}\right)\left(a_2^2+a_6^2-a_1^2-a_5^2\right)-3a_2^2a_6^2+3a_1^2a_5^2\geq0\\
    &\Leftrightarrow\sum_{i=1}^5a_i^2a_2^2+\sum_{i=1; i\neq2}^6a_i^2a_6^2-\sum_{i=1;i\neq5}^6a_i^2a_1^2-\sum_{i=2}^6a_i^2a_5^2-a_2^2a_6^2+a_1^2a_5^2+\sqrt{\Delta}\left(a_2^2+a_6^2-a_1^2-a_5^2\right)\geq0\\
    &\Leftarrow a_5^2a_2^2-a_6^2a_1^2+a_1 ^2a_6^2-a_2^2a_5^2-a_2^2a_6^2+a_1^2a_5^2+\sqrt{\Delta}\left(a_2^2+a_6^2-a_1^2-a_5^2\right)\geq0
\end{align*}
by comparing the terms of the first and the third sum ($a_2>a_1$) and those of the second and the fourth $(a_6>a_5)$.

So in order to have $\gamma\leq x_M$, it is sufficient to have this inequality:
\begin{equation}\label{ce qu'il nous faut}
    -a_2^2a_6^2+a_1^2a_5^2+\sqrt{\Delta}\left(a_2^2+a_6^2-a_1^2-a_5^2\right)\geq0.
\end{equation}
In our case we have:
\begin{align*}
    \Delta&=\sum_{i=1}^6a_i^4+2\left(a_1^2a_4^2+a_4^2a_5^2+a_5^2a_2^2+a_2^2a_3^2+a_3^2a_6^2+a_6^2a_1^2\right)\\
    -&\left(a_1^2a_5^2+a_1^2a_2^2+a_1^2a_3^2+a_4^2a_2^2+a_4^2a_3^2+a_4^2a_6^2+a_5^2a_3^2+a_5^2a_6^2+a_2^2a_6^2\right).
\end{align*}
Because $a_1^2a_5^2<a_6^2a_1^2$, $a_1^2a_2^2<a_1^2a_4^2$, $a_1^2a_3^2<a_1^2a_4^2$, $a_4^2a_2^2<a_4^2a_5^2$, $a_5^2a_3^2<a_3^2a_6^2$ and $a_2^2a_6^2<a_3^2a_6^2$, we have:
\begin{equation*}
    \Delta\geq \sum_{i=1}^6a_i^4+2\left(a_5^2a_2^2+a_2^2a_3^2\right)+a_4^2a_5^2+a_6^2a_1^2-\left(a_4^2a_3^2+a_4^2a_6^2+a_5^2a_6^2\right).
\end{equation*}
$(a-b-c)^2\geq0$ for all $a,b,c\in\R$ gives $ab+ac\leq \frac{\left(a^2+b^2+c^2\right)}{2}+bc$ for all $a,b,c\in\R$. So we have
\begin{equation*}
    a_4^2a_6^2+a_5^2a_6^2\leq\frac{1}{2}\left(a_6^4+a_4^4+a_5^4\right)+a_5^2a_4^2.
\end{equation*}
Hence, because $a_4^2a_3^2\leq\frac{1}{2}\left(a_3^4+a_4^4\right)$, we get:
\begin{align}\label{Delta minoration}
    \Delta&\geq \sum_{i=1}^6a_i^4+2\left(a_5^2a_2^2+a_2^2a_3^2\right)+a_6^2a_1^2-\frac{1}{2}\left(a_3^4+2a_4^4+a_5^4+a_6^4\right)\notag\\
    &\geq a_5^4+a_2^4+2a_2^2a_5^2=\left(a_2^2+a_5^2\right)^2.
\end{align}
As in the proof of \ref{Min des familles}, we use the notations $r_i$ where the $r_i$ are such that $r_i=a_i^2-a_{i-1}^2$ for all $i\in\llbracket 1,6\rrbracket$ with $a_0=0$ here.

So we have:
\begin{equation}\label{différence ri}
    a_1^2a_5^2-a_2^2a_6^2=r_1a_5^2-(r_1+r_2)\left(a_5^2+r_6\right)=-r_1r_6-r_2\left(\sum_{i=1}^6r_i\right).
\end{equation}
On the other hand, $a_2^2+a_6^2-a_1^2-a_5^2=r_2+r_6$. So by the equations \ref{Delta minoration} and \ref{différence ri}, we have
\begin{align*}
    &-a_2^2a_6^2+a_1^2a_5^2+\sqrt{\Delta}\left(a_2^2+a_6^2-a_1^2-a_5^2\right)\geq0\\
    \Leftarrow & -r_1r_6-r_2\left(\sum_{i=1}^6r_i\right)+(2r_1+2r_2+r_3+r_4+r_5)(r_2+r_6)\geq0\\
    \Leftrightarrow & ~r_6(r_1+r_2+r_3+r_4+r_5)+r_2(r_1+r_2)\geq0.
\end{align*}
The latter is always true because $r_i>0$ for all $i\geq2$ and $r_1\geq0$. Hence \ref{ce qu'il nous faut} holds and so $\gamma\leq x_M$. By Proposition \ref{delta}, we conclude that $\s{3}{5}{2}{4}{6}\not\in\mathcal{M}_6$.

\ding{176} Finally, we set $\sigma=\s{4}{5}{2}{3}{6}$ and $\mu=\s{3}{6}{2}{4}{5}$.
\begin{align*}
    \alpha&=a_1^2a_5^2+a_1^2a_2^2+a_1^2a_3^2+a_4^2a_2^2+a_4^2a_3^2+a_4^2a_6^2+a_5^2a_3^2+a_5^2a_6^2+a_2^2a_6^2\\
    -&\left(a_1^2a_6^2+a_1^2a_2^2+a_1^2a_4^2+a_3^2a_2^2+a_3^2a_4^2+a_3^2a_5^2+a_6^2a_4^2+a_6^2a_5^2+a_2^2a_5^2\right)\\
    &=a_1^2\left(a_5^2+a_3^2-a_4^2-a_6^2\right)+a_2^2\left(a_4^2+a_6^2-a_3^2-a_5^2\right)\\
    &=\left(a_2^2-a_1^2\right)\left(a_4^2+a_6^2-a_3^2-a_5^2\right)>0.
\end{align*}
On the other hand,
\begin{equation*}
    \beta=-a_1^2a_5^2a_3^2-a_4^2a_2^2a_6^2+a_1^2a_6^2a_4^2+a_3^2a_2^2a_5^2=\left(a_2^2-a_1^2\right)\left(a_5^2a_3^2-a_4^2a_6^2\right).
\end{equation*}
So $\gamma=\frac{a_4^2a_6^2-a_5^2a_3^2}{a_4^2+a_6^2-a_5^2-a_3^2}$ and with the same notations as above:
\begin{align*}
    &\gamma\leq x_M\Leftrightarrow \left(\sum_{i=1}^6a_i^2+\sqrt{\Delta}\right)\left(a_4^2+a_6^2-a_5^2-a_3^2\right)-3a_4^2a_6^2+3a_5^2a_3^2\geq0\\
    &\Leftrightarrow \sqrt{\Delta}\left(a_4^2+a_6^2-a_5^2-a_3^2\right)+a_5^2a_3^2-a_4^2a_6^2+\sum_{i=1}^5a_i^2a_4^2+\sum_{i=1;i\neq4}^6a_i^2a_6^2-\sum_{i=1;i\neq3}a_i^2a_5^2-\sum_{i=1;i\neq5}a_i^2a_3^2\geq0\\
    &\Leftarrow \sqrt{\Delta}\left(a_4^2+a_6^2-a_5^2-a_3^2\right)+a_5^2a_3^2-a_4^2a_6^2+a_4^2a_5^2+a_6^2a_3^2-a_4^2a_5^2-a_3^2a_6^2\geq0
\end{align*}
by comparing the terms of the first and the fourth sum ($a_4>a_3$) and those of the second and the third ($a_6>a_5$).

So in order to get $\gamma\leq x_M$, it is sufficient to have
\begin{equation}\label{ce qu'il nous faut 2}
    \sqrt{\Delta}\left(a_4^2+a_6^2-a_5^2-a_3^2\right)+a_5^2a_3^2-a_4^2a_6^2\geq 0.
\end{equation}
With the same notations as before, we have:
\begin{align*}
    a_5^2a_3^2-a_4^2a_6^2&=a_5^2a_3^2-\left(a_3^2+r_4\right)\left(a_5^2+r_6\right)\\
    &=-a_3^2r_6-r_4a_5^2-r_4r_6\\
    &=-r_6\left(\sum_{i=1}^4r_i\right)-r_4\left(\sum_{i=1}^5r_i\right).
\end{align*}
On the other hand, $a_4^2+a_6^2-a_5^2-a_3^2=r_4+r_6$. So, by the lower bound for $\Delta$ obtained in the inequality \ref{Delta minoration} ($\sigma$ is the same here as in the precedent case, so the value of $\Delta$ doesn't change), for having \ref{ce qu'il nous faut 2}, we only have to check that:
\begin{align*}
    &(2r_1+2r_2+r_3+r_4+r_5)(r_4+r_6)-r_6\left(\sum_{i=1}^4r_i\right)-r_4\left(\sum_{i=1}^5r_i\right)\geq0\\
    \Leftrightarrow&~r_4(r_1+r2)+r_6(r_1+r_2+r_5)\geq0.
\end{align*}
The latter is always true because $r_i>0$ for all $i\geq2$ and $r_1\geq0$. So we have \ref{ce qu'il nous faut 2} and so $\gamma\leq x_M$. By Proposition \ref{delta}, the proof is complete.
\end{proof}

\begin{thm}\label{final}
    $\mathcal{M}_6$ is composed by exactly the five following permutations:
    \begin{equation*}
        \s{5}{3}{4}{2}{6},\s{4}{5}{3}{2}{6},\s{5}{4}{3}{2}{6},\s{5}{4}{2}{3}{6},\s{4}{5}{2}{3}{6}.
    \end{equation*}
\end{thm}

\begin{proof}
    In view of Theorem \ref{Il en reste 5}, we just need to show that for every $\sigma$ above, $\sigma\in\mathcal{M}_6$. Thus, we will present five different weight sequences $a$ which will lead to each of the five possible cases. We will keep the notations of Proposition \ref{Gershgo} throughout the proof. Our goal is to use this proposition four times for each of the five weight sequences chosen in order to eliminate the permutations that are not minimizing the numerical radius. So we have to apply this proposition 20 times in the same way we did in the proof of \ref{Il en reste 5}. We will summarize the results of these computations in five tables.

$\sigma=\s{5}{3}{4}{2}{6}$ and we set here $a=(1, 1.3, 1.7, 6.3, 6.8, 7.1)$.

We have $g_\sigma=\min\{a_1+a_5,a_2+a_4\}$ and $G_\sigma=\max\{a_6+a_2,a_5+a_3\}$. Thus, with this weight sequence, $g_\sigma^2=57.76$ and $G_\sigma^2=72.25$.
\begin{center}
	\begin{tabular}{|c|c|c|}\hline
		 $\mu$&Sign of $\alpha$&Approximate value of $\gamma$\\\hline
		 $\s{4}{5}{3}{2}{6}$&$\alpha>0$&52.464\\\hline
		 $\s{5}{4}{3}{2}{6}$&$\alpha>0$&52.284\\\hline
		 $\s{5}{4}{2}{3}{6}$&$\alpha>0$&53.541\\\hline
		 $\s{4}{5}{2}{3}{6}$&$\alpha>0$&53.721\\\hline
	\end{tabular}
\end{center}

$\sigma=\s{4}{5}{3}{2}{6}$ and we set here $a=(1, 1.3, 1.4, 1.6, 2, 5)$.

We have $g_\sigma=\min\{a_1+a_4,a_2+a_3\}$ and $G_\sigma=\max\{a_5+a_4,a_6+a_2\}$. Thus, with this weight sequence, $g_\sigma^2=6.76$ and $G_\sigma^2=39.69$.
\begin{center}
	\begin{tabular}{|c|c|c|}\hline
		 $\mu$&Sign of $\alpha$&Approximate value of $\gamma$\\\hline
		 $\s{5}{3}{4}{2}{6}$&$\alpha<0$&45.915\\\hline
		 $\s{5}{4}{3}{2}{6}$&$\alpha<0$&49.281\\\hline
		 $\s{5}{4}{2}{3}{6}$&$\alpha>0$&-8.942\\\hline
		 $\s{4}{5}{2}{3}{6}$&$\alpha>0$&2.857\\\hline
	\end{tabular}
\end{center}

$\sigma=\s{5}{4}{3}{2}{6}$ and we set here $a=(1, 1.15, 1.22, 1.25, 1.3, 1.6)$.

We have $g_\sigma=\min\{a_1+a_5,a_2+a_3\}$ and $G_\sigma=\{a_6+a_2,a_5+a_4\}$. Thus, with this weight sequence, $g_\sigma^2=5.29$ and $G_\sigma^2=7.562$.
\begin{center}
	\begin{tabular}{|c|c|c|}\hline
		 $\mu$&Sign of $\alpha$&Approximate value of $\gamma$\\\hline
		 $\s{5}{3}{4}{2}{6}$&$\alpha<0$&8.174\\\hline
		 $\s{4}{5}{3}{2}{6}$&$\alpha>0$&5.094\\\hline
		 $\s{5}{4}{2}{3}{6}$&$\alpha>0$&2.771\\\hline
		 $\s{4}{5}{2}{3}{6}$&$\alpha>0$&3.39\\\hline
	\end{tabular}
\end{center}

$\sigma=\s{5}{4}{2}{3}{6}$ and we set here $a=(1, 3.4, 3.8, 4.2, 4.3, 4.5)$.

We have $g_\sigma=\min\{a_1+a_5,a_2+a_3\}$ and $G_\sigma=\max\{a_6+a_3,a_5+a_4\}$. Thus, with this weight sequence, $g_\sigma^2=28.09$ and $G_\sigma^2=72.25$.
\begin{center}
	\begin{tabular}{|c|c|c|}\hline
		 $\mu$&Sign of $\alpha$&Approximate value of $\gamma$\\\hline
		 $\s{5}{3}{4}{2}{6}$&$\alpha<0$&73.711\\\hline
		 $\s{4}{5}{3}{2}{6}$&$\alpha>0$&-201.887\\\hline
		 $\s{5}{4}{3}{2}{6}$&$\alpha<0$&136.698\\\hline
		 $\s{4}{5}{2}{3}{6}$&$\alpha>0$&20.8\\\hline
	\end{tabular}
\end{center}

$\sigma=\s{4}{5}{2}{3}{6}$ and we set here $a=(1, 1.03, 49.7, 53.5, 53.7, 54.7)$.

We have $g_\sigma=\min\{a_1+a_4,a_2+a_3\}$ and $G_\sigma=\max\{a_6+a_3,a_5+a_4\}$. Thus, with this weight sequence, $g_\sigma^2=2573.533$ and $G_\sigma^2=11491.84$.
\begin{center}
	\begin{tabular}{|c|c|c|}\hline
		 $\mu$&Sign of $\alpha$&Approximate value of $\gamma$\\\hline
		 $\s{5}{3}{4}{2}{6}$&$\alpha<0$&17008.607\\\hline
		 $\s{4}{5}{3}{2}{6}$&$\alpha<0$&78978.099\\\hline
		 $\s{5}{4}{3}{2}{6}$&$\alpha<0$&66430.751\\\hline
		 $\s{5}{4}{2}{3}{6}$&$\alpha<0$&11563.519\\\hline
	\end{tabular}
\end{center}
In these five tables, we always have $\gamma<g_\sigma^2$ when $\alpha>0$ and $\gamma>G_\sigma^2$ when $\alpha<0$. By Proposition \ref{Gershgo}, the proof is complete.

\end{proof}

\section{Remarks}

\begin{rem}\label{proportions}
    One might ask how often we "fall" into each case of Theorem \ref{final} for a random sequence of weights $a\in\mathcal{A}_6$. We generated 100 000 weight sequences $a\in\mathcal{A}_6$ randomly and then numerically found the permutation in $S_6/H_6$ which minimizes the numerical radius of $A_\sigma$. With the same order for the five permutations $\sigma\in\mathcal{M}_6$ as in the preceding theorem, we have $w(A_\sigma)< w(A_\mu), \forall \mu \in S_6/H_6,\mu\neq \sigma$ in 63.6\%, 12.8\%, 13.4\%, 10.2\% and 1.0\% of the cases.
\end{rem}

\begin{rem}
    It is impossible with our tools to determine the minimizing permutation for all the weights $a\in\mathcal{A}_6$. Indeed, with a test on 1 000 000 weight sequences $a\in\mathcal{A}_6$ generated randomly, our propositions \ref{delta} and \ref{Gershgo} succeed in determining the minimizing permutation in only 38\% of the cases.
\end{rem}

\begin{rem}
    By applying four times Proposition \ref{Gershgo}, one can determinate a "region" $R\subset\mathcal{A}_6$ such that for all $a\in R$, the minimizing permutation remains the same. But, as seen in the proof of Theorem \ref{final}, there are always two cases to consider for the value of $G_\sigma$ and the value of $g_\sigma$ for each of the five $\sigma$ which belong to $\mathcal{M}_6$. This leads to very complicated conditions to describe the regions, so it would not be very useful and concrete.
\end{rem}

\begin{rem}
    On the proof of Theorem \ref{final}, we can observe that when $\sigma=\s{5}{3}{4}{2}{6}$, $\alpha>0$ for the four different values of $\mu$ we consider. For the values of $a$ we consider in the same proof, when $\mu=\s{5}{3}{4}{2}{6}$, we always have $\alpha<0$ for all others $\sigma$. One might ask if it is the case for all $\mu\in S_6$ and all $a\in\mathcal{A}_6$. The answer to this question is positive as shown in the next proposition in view of Lemma \ref{id remar}.
\end{rem}

\begin{prop}\label{min coeff x2}
    Let $a\in\mathcal{A}_6$, the permutations $\sigma\in S_6$ that minimize the sum $\sum_{i=1}^6a_{\sigma(i)}^2a^2_{\sigma(i+1)}$ are exactly those of the form
    \begin{equation*}
        \s{5}{3}{4}{2}{6}h\text{ with }h\in H_6.
    \end{equation*}
\end{prop}

\begin{proof}
    By the discussion about $H_6$, we only need to show that if $\sigma\in S_6$ is such that $\sigma(1)=1$, $\sigma(2)<\sigma(6)$ and $\sigma$ minimize the value of $\sum_{i=1}^6a_{\sigma(i)}^2a^2_{\sigma(i+1)}$, then we have $\sigma=\s{5}{3}{4}{2}{6}$.

    So let $\sigma$ be such a permutation. If $\sigma$ minimizes the value of $\sum_{i=1}^6 a_{\sigma(i)}^2 a^2_{\sigma(i+1)}$, we have 
    \begin{equation}\label{min}
        \sum_{i=1}^6 a_{\sigma(i)}^2 a^2_{\sigma(i+1)}-a_{\mu(i)}^2 a^2_{\mu(i+1)}\leq0,
    \end{equation}
    for all $\mu\in S_6$ of the form $\mu=\tau\sigma$ with $\tau$ a transposition of the form $(\sigma(j),\sigma(j+1))$ or $(\sigma(j),\sigma(j+2))$ with $j\in\llbracket1,6\rrbracket$ if we set from now $\sigma(j)=\sigma(j_0)$ with $j_0\equiv j [6]$ and $1\leq j_0\leq 6$.

    With these notations, for all $j\in\llbracket1,6\rrbracket$, if $\mu=(\sigma(j),\sigma(j+1))\circ\sigma$, one can easily compute that:
    \begin{equation*}
        \sum_{i=1}^6 a_{\sigma(i)}^2 a^2_{\sigma(i+1)}-a_{\mu(i)}^2 a^2_{\mu(i+1)}=\left(a_{\sigma(j)}^2-a_{\sigma(j+1)}^2\right)\left(a_{\sigma(j-1)}^2-a_{\sigma(j+2)}^2\right).
    \end{equation*}
    Using the fact that $a\in\mathcal{A}_6$, the fact that $a_{\sigma(1)}<a_{\sigma(k)}$ for all $k\in\llbracket 2,6\rrbracket$ and the inequality \ref{min}, we get:
    \begin{itemize}
        \item with $j=1$: $\sigma(6)>\sigma(3)$;
        \item with $j=2$: $\sigma(2)>\sigma(3)$;
        \item with $j=5$: $\sigma(6)>\sigma(5)$;
        \item with $j=6$: $\sigma(2)>\sigma(5)$;
        \item with $j=3$: $\sigma(4)>\sigma(3)$ by using $\sigma(2)>\sigma(5)$;
        \item with $j=4$: $\sigma(4)>\sigma(5)$ by using $\sigma(6)>\sigma(3)$.
    \end{itemize}
    For having $\sigma$ completely determined, we need also the three relations $\sigma(3)/\sigma(5)$, $\sigma(2)/\sigma(4)$ and $\sigma(4)/\sigma(6)$ (we already have $\sigma(2)<\sigma(6)$ by hypothesis on $\sigma$).

    For all $j\in\llbracket1,6\rrbracket$, if $\mu=(\sigma(j),\sigma(j+2))\circ\sigma$, one can easily compute that:
    \begin{equation*}
        \sum_{i=1}^6 a_{\sigma(i)}^2 a^2_{\sigma(i+1)}-a_{\mu(i)}^2 a^2_{\mu(i+1)}=\left(a_{\sigma(j)}^2-a_{\sigma(j+2)}^2\right)\left(a_{\sigma(j-1)}^2-a_{\sigma(j+3)}^2\right).
    \end{equation*}
    The inequality \ref{min} then shows:
    \begin{itemize}
        \item with $j=2$: $\sigma(2)>\sigma(4)$ by using $\sigma(5)>\sigma(1)$;
        \item with $j=3$: $\sigma(3)>\sigma(5)$ by using $\sigma(6)>\sigma(2)$;
        \item with $j=4$: $\sigma(6)>\sigma(4)$ by using $\sigma(3)>\sigma(1)$.
    \end{itemize}
    We then have $\sigma=\s{5}{3}{4}{2}{6}$.
\end{proof}

\begin{rem}
    To solve the problem we had to use such tools as Gershgorin disks or derivation because the simplest ones would fail. Indeed, the permutation which maximizes the first of the two permutation-dependant coefficients in $P_{\sigma,a}$ that is $\left(a_{\sigma(1)}^2a_{\sigma(3)}^2+\cdots +a_{\sigma(4)}^2a_{\sigma(6)}^2\right)$ is always $\s{5}{3}{4}{2}{6}$ as shown in Proposition \ref{min coeff x2}. Unfortunately, it is also the permutation which minimizes for all $a\in\mathcal{A}_n$ the constant term of $P_{\sigma,a}$ that is $-a_{\sigma(1)}^2a_{\sigma(3)}^2a_{\sigma(5)}^2-a_{\sigma(2)}^2a_{\sigma(4)}^2a_{\sigma(6)}^2$ as the next proposition will show.
\end{rem}

\begin{prop}\label{min coeff cst}
    For all $a\in\mathcal{A}_6$, among the 10 families of Proposition \ref{Min des familles}, the family that minimizes the value of $-a_{\sigma(1)}^2a_{\sigma(3)}^2a_{\sigma(5)}^2- a_{\sigma(2)}^2a_{\sigma(4)}^2a_{\sigma(6)}^2$ is the family of $\s{5}{3}{4}{2}{6}$, namely the permutations in $S_6/H_6$ for which $\{\sigma(1),\sigma(3),\sigma(5)\}=\{1,2,3\}$.
\end{prop}

\begin{proof}
    Let $\sigma=\s{5}{3}{4}{2}{6}$ and $\mu\in S_6/H_6$ such that $\{\mu(1),\mu(3),\mu(5)\}\neq\{1,2,3\}$. With the same notation as in the propositions \ref{delta} and \ref{Gershgo}, we want to show that $\beta<0$.

    Let's denote by $E_1:=\{\mu(1),\mu(3),\mu(5)\}$ and $E_2:=\{\mu(2),\mu(4),\mu(6)\}$. By hypothesis on $\mu$, we know that one set among $E_1$ and $E_2$ shares 2 points $i_1,i_2$ with $\{1,2,3\}$ and the other one shares two other points $j_1,j_2$ with $\{4,5,6\}$. If we fix $i,j$ such that $\{1,2,3\}=\{i,i_1,i_2\}$ and $\{4,5,6\}=\{j,j_1,j_2\}$, we have $\{E_1,E_2\}=\left\{\{j,i_1,i_2\},\{i,j_1,j_2\}\right\}$. That leads to
    \begin{equation*}
        \beta=\left(a_{i_1}^2a_{i_2}^2-a_{j_1}^2a_{j_2}^2\right)\left(a_j^2-a_i^2\right).
    \end{equation*}
    By definition of $i,i_1,i_2,j,j_1,j_2$, we always have $\left(a_{i_1}^2a_{i_2}^2-a_{j_1}^2a_{j_2}^2\right)<0$ and $\left(a_j^2-a_i^2\right)>0$. Hence $\beta<0$.
\end{proof}

\begin{rem}
    The propositions \ref{min coeff x2} and \ref{min coeff cst} and the remarks \ref{simple impossible} and \ref{proportions} shed light on the special role played by $\s{5}{3}{4}{2}{6}$ in our problem. In fact, this permutation also plays a special role for $n\neq6$ in the following sense. We have:
    \begin{equation*}
        \mathcal{M}_4=\left\{\compactpmatrix{
            1&2&3&4\\
            1&3&2&4}\right\}\text{ and }\mathcal{M}_5=\left\{\compactpmatrix{
            1&2&3&4&5\\
            1&4&3&2&5}\right\}.
    \end{equation*}
    For $n>6$, numerically the same pattern seems to appear in the shape of the "most successful" member of $\mathcal{M}_n$ (in the sense of Remark \ref{proportions}). This "most successful" permutation seems to be:
    \begin{align*}
        &(2,n-1) \text{ if }n-1>2\text{ and }n-3\leq4,\textit{ ie. }n=4,5,6,7;\\
        &(2,n-1)\circ(4,n-3) \text{ if }n-3>4\text{ and }n-5\leq6,\textit{ ie. }n=8,9,10,11;\\
        &(2,n-1)\circ(4,n-3)\circ(6,n-5) \text{ if }n-5>6\text{ and }n-7\leq8,\textit{ ie. }n=12,13,14,15;\\
        &\cdots
    \end{align*}
    The permutation $\sigma_M^{(a)}$ we give in the introduction is such that the big weights are close to each other and the small weights are close to each other. It is the contrary for our "most successful" permutation, the weights are such that we have one big weight following one small weight following one big weight and so on. For maximizing the numerical radius of $A_\sigma$, $\sigma$ needs to group the big weights in $a$ while for having $w(A_\sigma)$ small, a good way is for $\sigma$ to scatter the big weights in $a$ among the small ones.

    This most successful permutation also seems to verify a generalisation of Proposition \ref{min coeff x2} for all $n>6$. 
\end{rem}

\begin{rem}
    The techniques used in \cite{chang2012maximizing} and \cite{gau2024proof} use the fact that the eigenvector $x_M$ of $\RE(A_{\sigma_M})$ corresponding to the largest eigenvalue (so the numerical radius of $A_{\sigma_M}$ in our case) has only positive coefficients and the way they are ordered in $x_M$ follow almost $\sigma_M$. For the minimum, the coefficients of the eigenvector are still positive, but the way they are ordered in the eigenvector seems to us not to be linked with the minimizing permutation when it's unique.
\end{rem}

\begin{rem}
    One might ask, what is the situation for $n>6$. The cardinal of $\mathcal{M}_n$ with $n>6$ becomes too big for using the tools developed here and to understand the minimum case in the way of Theorem \ref{final}. For example, by numerical computations\footnote{We generate 100 000 weights $a\in \mathcal{A}_7$ randomly and build a subset of $\mathcal{M}_7$ with the different minimizing permutations computed.}, the cardinal of $\mathcal{M}_7$ is at least 51. It seems too big to say anything concrete.
\end{rem}

\section*{Acknowledgments}
The author acknowledges the support of the CDP C2EMPI, together
with the French State under the France-2030 programme, the University of Lille,
the Initiative of Excellence of the University of Lille, the European Metropolis of
Lille for their funding and support of the R-CDP-24-004-C2EMPI project.

\section*{Declaration of competing interest}

No competing interest.

\printnomenclature

\bibliography{References.bib}

@book{gau2021numerical,
  title={Numerical ranges of Hilbert space operators},
  author={Gau, Hwa-Long and Wu, Pei Yuan},
  volume={179},
  year={2021},
  publisher={Cambridge University Press}
}

@article{gau2024proof,
  title={Proof of a conjecture on numerical ranges of weighted cyclic matrices},
  author={Gau, Hwa-Long},
  journal={Linear Algebra and its Applications},
  volume={682},
  pages={295--308},
  year={2024},
  publisher={Elsevier}
}

@article{chien2023numerical,
  title={Numerical ranges of cyclic shift matrices},
  author={Chien, Mao-Ting and Kirkland, Steve and Li, Chi-Kwong and Nakazato, Hiroshi},
  journal={Linear Algebra and its Applications},
  volume={678},
  pages={268--294},
  year={2023},
  publisher={Elsevier}
}

@article{chang2012maximizing,
  title={Maximizing numerical radii of weighted shifts under weight permutations},
  author={Chang, Chi-Tung and Wang, Kuo-Zhong},
  journal={Journal of Mathematical Analysis and Applications},
  volume={394},
  number={2},
  pages={592--602},
  year={2012},
  publisher={Elsevier}
}

@article{li2002numerical,
  title={The numerical range of a nonnegative matrix},
  author={Li, Chi-Kwong and Tam, Bit-Shun and Wu, Pei Yuan},
  journal={Linear algebra and its applications},
  volume={350},
  number={1-3},
  pages={1--23},
  year={2002},
  publisher={Elsevier}
}
\bibliographystyle{alpha}
\end{document}